\documentclass[11pt]{amsart}
\usepackage{amsmath, amssymb, array}
\usepackage[all]{xy}
\usepackage{ascmac}
\usepackage[dvips]{graphicx}
\usepackage{color}
\usepackage{amscd}
\usepackage[driverfallback=dvipdfm]{hyperref}
\usepackage{amsrefs}
\usepackage{cleveref}
\usepackage{shuffle}
\usepackage{mathrsfs}
\usepackage{listings} 
\usepackage{comment}

\newfont{\cyr}{wncyr10 scaled\magstep0}

\newcommand{\C}{\mathbb{C}}

\newcommand{\F}{\mathbb{F}}

\newcommand{\PP}{\mathbb{P}}
\newcommand{\Q}{\mathbb{Q}}
\newcommand{\R}{\mathbb{R}}

\newcommand{\Z}{\mathbb{Z}}

\DeclareMathOperator{\rad}{rad} 

\date{\today}

\theoremstyle{plain}
 \newtheorem{theorem}{Theorem}[section]
 \crefname{theorem}{Theorem}{Theorems}
 \newtheorem{proposition}[theorem]{Proposition}
 \crefname{proposition}{Proposition}{Propositions}
 \newtheorem{lemma}[theorem]{Lemma}
 \crefname{lemma}{Lemma}{Lemmas}
 \newtheorem{corollary}[theorem]{Corollary}
 \crefname{corollary}{Corollary}{Corollaries}
 \newtheorem{conjecture}[theorem]{Conjecture}
 \crefname{conjecture}{Conjecture}{Conjectures}

 \crefname{question}{Question}{Questions}
 
 \crefname{problem}{Problem}{Problems}

\theoremstyle{definition} 
 
 \crefname{definition}{Definition}{Definitions}
 
 \crefname{example}{Example}{Examples}
 \newtheorem{remark}[theorem]{Remark}
 \crefname{remark}{Remark}{Remarks}

\title{Primes of the form $X^{3}+NY^{3}$ and a family of non-singular plane curves which violate the local-global principle
}
\author{Yoshinosuke Hirakawa}
\address[Yoshinosuke Hirakawa]{Department of Mathematics \\ Faculty of Science and Technology \\ Tokyo University of Science, 2641, Yamazaki, Noda, Chiba, Japan}
\email{hirakawa\_yoshinosuke@ma.noda.tus.ac.jp}
	\thanks{This research was supported in part by KAKENHI 18H05233.}
\date{\today}

\begin{document}


\maketitle


\begin{abstract}
Let $n$ be an integer such that $n = 5$ or $n \geq 7$.
In this article, 
we introduce a recipe for a certain infinite family of non-singular plane curves of degree $n$
which violate the local-global principle.
Moreover,
each family contains infinitely many members which are not geometrically isomorphic to each other.
Our construction is based on two arithmetic objects;
that is, prime numbers of the form $X^{3}+NY^{3}$ due to Heath-Brown and Moroz
and the Fermat type equation of the form $x^{3}+Ny^{3} = Lz^{n}$,
where $N$ and $L$ are suitably chosen integers.
In this sense, our construction is an extension of the family of odd degree $n$
which was previously found by Shimizu and the author.
The previous construction works only if the given degree $n$ has a prime divisor $p$
for which the pure cubic fields $\Q(p^{1/3})$ or $\Q((2p)^{1/3})$
satisfy a certain indivisibility conjecture of Ankeny-Artin-Chowla-Mordell type.
In this time, we focus on the complementary cases,
namely the cases of even degrees and exceptional odd degrees.
Consequently, our recipe works well as a whole.
This means that we can unconditionally produce infinitely many explicit non-singular plane curves of every degree $n = 5$ or $n \geq 7$
which violate the local-global principle.
This gives a conclusion of the classical story of searching explicit ternary forms violating the local-global principle,
which was initiated by Selmer (1951) and extended by Fujiwara (1972) and others.
\end{abstract}



\section{Introduction}

In the theory of Diophantine equations,
the local-global principle for quadratic forms established by Minkowski and Hasse
is one of the major culminations (cf.\ \cite[Theorem 8, Ch. IV]{Serre_Arithmetic}).
It states that
every quadratic hypersurface of the projective space $\PP^{n}$ $(n \geq 1)$ defined over $\Q$
has a rational point over $\Q$
if (and only if) it has a rational point over $\R$ and $\Q_{p}$ for every prime number $p$.
Here and after,
$\Q, \R, \Q_{p}$ denote as usual the field of rational, real, and $p$-adic numbers respectively.

In contrast to the quadratic case,
there exist many homogeneous forms of higher degrees which violate the local-global principle.
For example, Selmer \cite{Selmer} found that a non-singular plane cubic curve defined by
\begin{equation} \label{Selmer}
	3X^{3}+4Y^{3} = 5Z^{3}, \quad \text{or equivalently} \quad X^{3}+6Y^{3} = 10Z^{3}
\end{equation}
has rational points $\R$ and $\Q_{p}$ for every prime number $p$
but not over $\Q$.
In this situation,
we say that (the curve defined by) \cref{Selmer} violates the local-global principle
or it is a counterexample to the local-global principle.
From \cref{Selmer},
we can easily construct reducible (especially singular) counterexamples of higher degrees.
Note also that in the case of weighted homogeneous forms,
Lind \cite{Lind} and Reichardt \cite{Reichardt} independently found the following first counterexample
\[
	X^{4}-17Y^{4} = 2Z^{2}.
\]
For more information on this topic,
we refer the reader to a nice introductory article \cite{Aitken-Lemmermeyer} by Aitken and Lemmermeyer.

More recently,
some people have studied infinite families of non-singular plane curves
which violate the local-global principle.
The existence of such a family of cubic curves was proven by Colliot-Th\'el\`ene and Poonen in \cite{CTPoonen},
and an explicit example was constructed first by Poonen in \cite{Poonen_genus-one} as follows.

\begin{theorem} [\cite{Poonen_genus-one}] \label{Poonen_genus_one}
For any $t \in \Q$, the equation
\[
	5X^{3}+9Y^{3}+10Z^{3}+12 \left( \frac{t^{2}+82}{t^{2}+22} \right)^{3} (X+Y+Z)^{3} = 0
\]
defines a non-singular plane cubic curve $C = C_{t}$ defined over $\Q$ which violates the local-global principle.
Moreover, there exists a set of $t \in \Q$
which gives infinitely many geometrically non-isomorphic classes of such curves.
\end{theorem}

In contrast to the cubic case,
it seems that there is only a little number of explicit non-singular plane curves of a higher odd degree which violates the local-global principle.
After the above examples by Lind, Reichardt, and Selmer,
Fujiwara \cite{Fujiwara} found that a non-singular plane quintic curve defined by
\begin{equation} \label{Fujiwara_example}
	(X^{3}+5Z^{3})(X^{2}+XY+Y^{2}) = 17Z^{5}
\end{equation}
violates the local-global principle.
His idea is to reduce the proof of the unsolubility over $\Q$
to the determination of primitive solutions of the Fermat type equation $x^{3}+5y^{3} = 17z^{5}$.
Here and after, a triple of integers $(x, y, z) \in \Z^{\oplus 3}$ is called primitive if $\gcd(x, y, z) = 1$.
After Fujiwara,
Cohen \cite[Corollary 6.4.11]{Cohen} gave several counterexamples of the form $x^{p}+by^{p}+cz^{p} = 0$
of degree $p = 3, 5, 7, 11$ with $b, c \in \mathbb{Z}$,
but his construction is still so restricted.
More recently,
Shimizu and the author conditionally succeeded in generalizing Fujiwara's idea in general higher odd degrees.
Our construction is based on the generation of prime numbers by cubic polynomials established by Heath-Brown and Moroz \cite{HBM2002,HBM2004}.

\begin{theorem} [{\cite[Theorem 1]{HBM2004}}] \label{HBM}
Let $f_{0} \in \mathbb{Z}[X, Y]$ be an irreducible binary cubic form,
$\rho \in \mathbb{Z}$,
$(\gamma_{1}, \gamma_{2}) \in \mathbb{Z}^{\oplus 2}$,
and $\gamma_{0}$ be
the greatest common divisor of the coefficients of
$f_{0}(\rho x+\gamma_{1}, \rho y+\gamma_{2})$.
Set $f(x, y) := \gamma_{0}^{-1}f_{0}(\rho x+\gamma_{1}, \rho y+\gamma_{2})$.
Suppose that $\gcd(f(\mathbb{Z}^{\oplus 2})) = 1$.
Then, the set $f(\mathbb{Z}^{\oplus 2})$
contains infinitely many prime numbers.
\end{theorem}

By using the above result,
Shimizu and the author proved the following theorem.

\begin{theorem} [{cf.\ \cite[Theorem 1.1]{Hirakawa-Shimizu_arXiv} and \cref{outline}}] \label{Hirakawa-Shimizu}
Let $p$ be an odd prime number
and $P = 2p \ \text{or} \ p$ so that $P \not\equiv \pm 1 \bmod{9}$.
Let $\epsilon = \alpha + \beta P^{1/3} + \gamma P^{2/3} \in \R_{> 1}$
be the fundamental unit of $\Q(P^{1/3})$
with $\alpha, \beta, \gamma \in \mathbb{Z}$.
Set
\[
	\iota =
	\begin{cases}
	1 & \text{if $\beta \not\equiv 0 \bmod{p}$ or $\beta \equiv \gamma \equiv 0 \bmod{p}$} \\
	2 & \text{if $\beta \equiv 0 \bmod{p}$ and $\gamma \not\equiv 0 \bmod{p}$}
	\end{cases}.
\]
Let $n \geq 5$ be an odd integer divisible by $p^{\iota}$.
Then, there exist
infinitely many $(n-3)/2$-tuples of pairs of integers $(b_{j}, c_{j})$ $(1 \leq j \leq (n-3)/2)$
satisfying the following condition:

There exist infinitely many integers $L$ such that
the equation
\begin{equation} \label{equation_Hirakawa-Shimizu}
	(X^{3}+P^{\iota}Y^{3}) \prod_{j = 1}^{\frac{n-3}{2}} (b_{j}^{2}X^{2}+b_{j}c_{j}XY+c_{j}^{2}Y^{2}) = LZ^{n}
\end{equation}
define non-singular plane curves of degree $n$ which violate the local-global principle.

Moreover, for each $n$,
there exists a set of such $(n-3)/2$-tuples $((b_{j}, c_{j}))_{1 \leq j \leq (n-3)/2}$
which gives infinitely many geometrically non-isomorphic classes of such curves of degree $n$.
\end{theorem}

Here, note that for $p = 3$,
we can take $P = 6$ so that $\beta \equiv \gamma \equiv 0 \bmod{3}$, hence $\iota = 1$ (cf.\ \cref{outline}).
On the other hand,
Shimizu and the author \cite{Hirakawa-Shimizu_arXiv} formulated the following conjecture,
which implies again $\iota = 1$ for all $p \neq 3$.

\begin{conjecture} \label{AACM_cubic}
Let $p \neq 3$ be a prime number,
$P = p \ \text{or} \ 2p$,
and $\epsilon = \alpha+\beta p^{1/3}+\gamma p^{2/3} \in \R_{> 1}$
be the fundamental unit of $\Q(P^{1/3})$
with $\alpha, \beta, \gamma \in (1/3)\mathbb{Z}$.
Then, we have $\beta \not\equiv 0 \bmod{p}$.
\end{conjecture}

The above \cref{AACM_cubic} for pure cubic fields $\Q(P^{1/3})$ may be regarded as an analogue of the classical Ankeny-Artin-Chowla-Mordell conjecture for real quadratic fields $\Q(p^{1/2})$ (cf.\ \cite{AAC,Mordell_Pellian_II}).
In fact, the authors verified in \cite{Hirakawa-Shimizu_arXiv}
that \cref{AACM_cubic} for all $p < 10^{5}$ by numerical examination using Magma \cite{Magma}.
Anyway, \cref{Hirakawa-Shimizu} gives, in a certain uniform manner, an explicit construction
of infinitely many non-singular plane curves which violate the local-global principle,
but it works only conditionally for higher odd degrees.
As far as the author knows,
this is the best result on explicit construction of non-singular plane curves of a general odd degree
which violate the local-global principle.

On the other hand,
in even degree case,
after Fujiwara's counterexamples of degree $5$,
the first counterexamples of degree $4$ were found by Bremner-Lewis-Morton \cite{Bremner-Lewis-Morton}
and Schinzel \cite{Schinzel_Hasse} independently,
whose equations are given as follows respectively:
\[
	3X^{4}+4Y^{4} = 19Z^{4} \quad \text{and} \quad
	X^{4}-2Y^{4}-16Y^{2}Z^{2}-49Z^{4} = 0.
\]
Moreover,
after Poonen's family of counterexamples of degree $3$,
Nguyen \cite{Nguyen_AMS,Nguyen_QJM,Nguyen_Tokyo} gave infinitely many counterexamples of almost all even degrees.
\footnote{
In fact, the recipe in \cref{Nguyen_QJM} actually gives explicit counterexamples of degree $4k+2$ for every $k \geq 1$.
On the other hand,
Nguyen \cite{Nguyen_Tokyo} asserted only for $k \neq 1, 2, 4$ that the recipe in \cref{Nguyen_Tokyo}
actually gives explicit counterexamples of degree $4k$.
See also \cite{Nguyen_AMS} for an explicit infinite family of counterexamples of degree $4$.
	}
Although it is difficult to describe his result in complete detail,
we recall here a part of them as follows:

\begin{theorem} [{\cite[Theorem 1.4]{Nguyen_QJM}}] \label{Nguyen_QJM}
Let $n = 2k$ be an even integer with $k \geq 1$.
Take $p, d, m, \text{and} \ \alpha$ as follows:
\begin{enumerate}
\item
$p$ is a prime number such that $p \equiv 1 \bmod{8}$.

\item
$d$ is an integer which is a quadratic non-residue in $\F_{p}^{\times}$ and prime to $n$.

\item
$m$ is an even integer such that $q := d^{2}+pm^{2}$ is a prime number.

\item
$\alpha$ is a rational number such that
$\alpha \in \Z_{l}$ for every prime divisor $l$ of $dp$
and $\alpha \neq 0, qp^{-k}, qd^{-k}, (m(d+p)-2q)((dp)^{k}-d^{k}-p^{k})^{-1}$.
\end{enumerate}
Set
$A = q-\alpha p^{k}$,
$B = q-\alpha d^{k}$,
and $C = m(d+p)-2q-\alpha((dp)^{k}-d^{k}-p^{k})$.
Then, the equation
\begin{align*}
	& pq^{2}X^{4k+2} + Y^{4k-2}(d(d+p)X^{2}-qY^{2})(pm^{2}(d+p)X^{2}-dqY^{2}) \\
	&\quad - Z^{2}(AX^{2k} + BY^{2k} + CX^{k}Y^{k} + \alpha Z^{2k})^{2} = 0
\end{align*}
defines a plane curve $C_{p, d, m, \alpha}$ of degree $4k+2$ over $\Q$
which violates the local-global principle.
\end{theorem}

\begin{theorem} [{a specialized form at $k = 0$ of \cite[Theorem 3.1]{Nguyen_Tokyo}}] \label{Nguyen_Tokyo}
Let $m, n \geq 1$ be integers such that $m < n$.
Let $p$ be a prime number and $\alpha, \beta \in \Z$.
Define a homogeneous polynomial $d(X, Y) \in \Z[X, Y]$ by
\[
	d(X, Y) = X^{m}(p(X+Y))^{n} - (-Y)^{m}(p(X+Y)+Y)^{n}.
\]
Suppose that
\begin{enumerate}
\item
$p \equiv 1 \bmod{8}$.

\item
$\gcd(\alpha\beta, p) = 1$.

\item
Every odd prime divisor of $\alpha d(\alpha, \beta)$ is a quadratic residue in $\F_{p}^{\times}$.

\item
$\beta$ is a quadratic non-residue in $\F_{p}^{\times}$.

\item
Every odd prime divisor of $\beta^{2}+p(\alpha+\beta)^{2}$ is prime to $n-m$.
\end{enumerate}
Then, there exists a family of explicit homogeneous polynomials $Q_{\zeta} = Q_{\zeta}(X, Y, Z) \in \Q[X, Y, Z]$ of degree $2n-1$
parametrized by a rational number $\zeta \in \Q$ such that
the equation
\[
	Q_{\zeta}(X, Y, Z)^{2}Z^{2}
	= p(\alpha X^{2n} + \beta Y^{2n})^{2}
		+ Y^{4n-4m}((p\alpha + (p+1)\beta)X^{2m} - p(\alpha+\beta)Y^{2m})^{2}
\]
defines a plane curve $C_{p, \alpha, \beta, \zeta}$ of degree $4n$ over $\Q$
which violates the local-global principle.
\end{theorem}

A remarkable character which the families obtained by Nguyen share is that
each member of these families covers a hyperelliptic curve which violates the local-global principle,
and the latter violation of the local-global principle is explained by the Brauer-Manin obstruction by a certain Brauer class of degree $2$.


As the first main theorem of this article,
we introduce another family of non-singular plane curves of even degrees $n \geq 8$
which violate the local-global principle.

\begin{theorem} [First main theorem] \label{main_even}
Let $u$ be an odd square-free integer such that $u \not\equiv \pm4 \bmod{9}$.
Set $P = 2u$.
Let $\epsilon = \alpha + \beta P^{1/3} + \gamma P^{2/3} \in \R_{> 1}$
be the fundamental unit of $\Q(P^{1/3})$
with $\alpha, \beta, \gamma \in \mathbb{Z}$.
Suppose that $\beta$ is even and the class number of $\Q(P^{1/3})$ is odd.
Let $n \geq 8$ be an even integer,
and $m \geq 3$ be an odd integer such that $m < n$.
Then, there exist
infinitely many $(n-6)/2$-tuples of pairs of integers $(b_{j}, c_{j})$ $(1 \leq j \leq (n-6)/2)$
satisfying the following condition:

There exist infinitely many prime numbers $l$ and infinitely many pairs of integers $(b_{0}, c_{0})$ such that
the equation
\begin{equation} \label{equation_even}
	(X^{3}+P^{2}Y^{3})(b_{0}X^{3}+lc_{0}Y^{3}) \prod_{j = 1}^{(n-6)/2} (b_{j}^{2}X^{2}+b_{j}c_{j}XY+c_{j}^{2}Y^{2}) = l^{m}Z^{n}
\end{equation}
define non-singular plane curves of degree $n$ which violate the local-global principle.

Moreover, for each $n$,
there exists a set of such $(n-4)/2$-tuples $((b_{j}, c_{j}))_{0 \leq j \leq (n-6)/2}$
which gives infinitely many geometrically non-isomorphic classes of such curves of degree $n$.
\end{theorem}

Here, note that the plane curve defined by \cref{equation_even} covers a hyperelliptic curve defined by
\begin{equation} \label{equation_hyp}
	(X^{3}+P^{2}Y^{3})(b_{0}X^{3}+lc_{0}Y^{3}) \prod_{j = 1}^{(n-6)/2} (b_{j}^{2}X^{2}+b_{j}c_{j}XY+c_{j}^{2}Y^{2}) = l^{m}Z^{2}.
\end{equation}
This geometric situation is similar to Nguyen's examples in \cref{Nguyen_QJM,Nguyen_Tokyo}.
However, we demonstrate in \S6 that the above hyperelliptic curve may has a $\Q$-rational point.
In particular, the violation of the local-global principle for the plane curve defined by \cref{equation_even}
cannot be explained by the Brauer-Manin obstruction on the hyperelliptic curve defined by \cref{equation_hyp} in general.
This arithmetic situation is in contrast to Nguyen's examples.

On the other hand,
as its visual suggested,
the family in \cref{main_even} is a variant of the family of odd degrees in \cref{Hirakawa-Shimizu}.
In fact, their proofs are quite similar.
For example,
the key ingredients of the proof of \cref{main_even} are the generation of prime numbers of the form $X^{3}+P^{2}Y^{3}$ (cf.\ \cref{HBM})
and a property of the primitive solutions of the Fermat type equation $x^{3}+P^{2}y^{3} = l^{m}z^{n}$.
Moreover, the proof of the infinitude of geometric isomorphy is essentially given in the previous work \cite[Lemma 4.1]{Hirakawa-Shimizu_arXiv}.

It should be noted, however, that
if we replace
the cubic form $X^{3}+P^{2}Y^{3}$ in \cref{main_even} to
the original cubic forms $X^{3}+p^{\iota}Y^{3}$ or $X^{3}+(2p)^{\iota}Y^{3}$ in \cref{Hirakawa-Shimizu} for ``$p = 2$",
then the proof of \cref{main_even} is broken down.
In fact,
we need an important but implicit condition under which the proofs of \cref{Hirakawa-Shimizu,main_even} work, that is,
for the fundamental unit $\epsilon$ of $\Q(P^{1/3})$,
the modulo $p^{\iota}$ class of $\epsilon^{p-1}$ has order at most $p$,
which is automatic whenever $p \geq 3$.
However, the modulo $2^{\iota}$ class of the fundamental unit $\epsilon = 1+2^{1/3}+2^{2/3}$ of $\Q(2^{1/3})$ has order $2^{\iota+1}$ for every $\iota \geq 1$.
The condition on the parity of $\beta$ in \cref{main_even} ensures that the above argument works well.
Therefore,
the existence of an auxiliary integer $u$ satisfying the conditions in \cref{main_even},
for example $u = 3, 7, 17, 21, 35, 39, \dots$, is essential in order to ensure that the recipe in \cref{main_even} works well for all even degrees $n \geq 8$ unconditionally.
For the arithmetic background of this subtlety, see \cref{Fujiwara_Fermat_unify,infinite_cubic_forms}.

In order to explain the second main theorem,
we mention another interesting aspect of \cref{main_even}.
Note that although the cubic form $X^{3}+P^{2}Y^{3}$ in \cref{main_even} suggests the analogy with the case of $\iota = 2$ in \cref{Hirakawa-Shimizu},
we do not need to assume that $n$ is divisible by $4 = 2^{2}$.
This is justified by a key ingredient \cref{Fujiwara_Fermat_unify} (see also \cref{release}).
In fact, by combining \cref{Fujiwara_Fermat_unify} with the existing result \cref{Hirakawa-Shimizu},
we obtain the following second main theorem in odd degree case.
Here, note that although the appearance of \cref{main_unify} is almost the same as that of \cref{Hirakawa-Shimizu},
the crucial difference is that, in \cref{main_unify},
we assume only that $n$ is divisible by $p$ but not necessarily divisible by $p^{\iota}$.
This releases us from the highly mysterious \cref{AACM_cubic}.

\begin{theorem} [Second main theorem] \label{main_unify}
Let $p$ be an odd prime number
and $P = 2p \ \text{or} \ p$ so that $P \not\equiv \pm 1 \bmod{9}$.
Let $\epsilon = \alpha + \beta P^{1/3} + \gamma P^{2/3} \in \R_{> 1}$
be the fundamental unit of $\Q(P^{1/3})$
with $\alpha, \beta, \gamma \in \mathbb{Z}$.
Set
\[
	\iota =
	\begin{cases}
	1 & \text{if $\beta \not\equiv 0 \bmod{p}$ or $\beta \equiv \gamma \equiv 0 \bmod{p}$} \\
	2 & \text{if $\beta \equiv 0 \bmod{p}$ and $\gamma \not\equiv 0 \bmod{p}$}
	\end{cases}.
\]
Let $n \ge 5$ be an odd integer divisible by $p$.
Then, there exist infinitely many $(n-3)/2$-tuples of pairs of integers $(b_{j}, c_{j})$ $(1 \leq j \leq (n-3)/2)$
satisfying the following condition:

There exist infinitely many integers $L \in \Z$ such that the equation
\begin{equation} \label{equation_unify}
	(X^{3}+P^{\iota}Y^{3}) \prod_{j = 1}^{(n-3)/2} (b_{j}^{2}X^{2}+b_{j}c_{j}XY+c_{j}^{2}Y^{2}) = LZ^{n}
\end{equation}
define non-singular plane curves of degree $n$ which violate the local-global principle.

Moreover, for each $n$,
there exists a set of such $(n-3)/2$-tuples $((b_{j}, c_{j}))_{1 \leq j \leq (n-3)/2}$
which gives infinitely many geometrically non-isomorphic classes of such curves of degree $n$.
\end{theorem}

As a consequence,
we obtain the following conclusion:

\begin{corollary}
For every integer $n$ such that $n = 5$ or $n \geq 7$,
there exist infinitely many non-singular plane curves
of the form \cref{equation_even} or \cref{equation_unify} according to the parity of $n$
which violate the local-global principle.
\end{corollary}

Here, we should emphasize that
although it is unclear from the above statements,
the proofs of \cref{Hirakawa-Shimizu,main_unify} (resp.\ \cref{main_even}) show that
for every odd (resp.\ even) integer $n$ such that $n = 5$ or $n \geq 7$,
we have an algorithm to produce arbitrarily many explicit parameters $(b_{j}, c_{j})$ and $L$ (resp.\ $l$)
for which \cref{equation_unify} (resp.\ \cref{equation_even}) define non-singular plane curves of degree $n$
which violate the local-global principle.
For numerical examples, see \cite[\S5]{Hirakawa-Shimizu_arXiv} and \S6 of the present article respectively.
Note also that for exceptionally small degrees $n = 3, 4, 6$,
we already have an algorithm due to Poonen \cite{Poonen_genus-one} and Nguyen \cite{Nguyen_AMS,Nguyen_QJM}.
(For non-singularity of Nguyen's curves of degree 6 in \cite{Nguyen_QJM},
see \S7 of this article.)
Therefore,
for every arbitrarily given degree $n \geq 3$,
we now obtain an algorithm to produce arbitrarily many explicit non-singular plane curves of degree $n$
which violate the local-global principle.
This gives a conclusion of the classical story of searching explicit ternary forms violating the local-global principle,
which was initiated by Selmer and extended by Fujiwara and others.

We conclude the introduction by presenting the organization of this article.
In \S2, we give a recipe
which exhibits how to construct counterexamples to the local-global principle defined by \cref{equation_even}
from prime numbers of the form $X^{3}+P^{2}Y^{3}$ and the Fermat type equations $x^{3}+P^{2}y^{3} = Lz^{n}$.
In \S3, we reduce the proof of \cref{main_even} to the key \cref{Fermat_unify} by using \cref{HBM}.
By an exactly similar manner,
we reduce the proof of \cref{main_unify} to the key \cref{Fermat_unify}.
The proof of the key \cref{Fermat_unify} itself is given in \S4,
where we again use \cref{HBM}.
In \S5, we give a variant \cref{main_deg4} of \cref{main_even} focusing on degrees divisible by $4$.
In \S6, we demonstrate how our construction in \cref{main_even} works for each given degree.
We give two numerical examples of degree $8$ following the proofs of \cref{main_even,main_deg4} respectively.
It should be emphasized that both of these examples are $\Z/4\Z$-coverings of hyperelliptic curves with $\Q$-rational points.
This means that the violation of the local-global principle for the former plane curves cannot be explained
by the Brauer-Manin obstruction on the latter hyperelliptic curves,
which is in contrast to Nguyen's examples in \cref{Nguyen_QJM,Nguyen_Tokyo}.
In \S7, we give a proof of non-singularity of a sub family of Nguyen's plane curves of degree 6,
which is implicit in the original article \cite{Nguyen_QJM}.

\section{Construction from prime numbers and Fermat type equations}

Let $p = 2$, $u$ be an odd integer, $P = 2u$, and $\iota = 1 \ \text{or} \ 2$.
We fix these integers throughout this section.
In this section, we prove the following proposition,
which gives explicit counterexamples to the local-global principle of even degree $n$
under the assumption that we have
\begin{itemize}
\item
	sufficiently many prime numbers of the form $P^{\iota}b^{3}+c^{3}$ with $b, c \in \Z$ and
	
\item
	integers $L$ such that the equation $x^{3}+P^{\iota}y^{3} = Lz^{n}$ has a specific property.
\end{itemize}
In what follows,
for each prime number $l$,
$v_{l}(n)$ denotes the additive $l$-adic valuation of $n \in \Z$.
Furthermore, for every non-zero integer $L$,
we denote the radical of $L$ by $\rad L := \prod_{l} l$,
where $l$ runs over the prime numbers such that $v_{l}(L) \geq 1$.

\begin{proposition} \label{recipe_even}
Let $n$ be an even integer such that $n \geq 8$.
Let $b_{j}, c_{j}$ $(0 \leq j \leq (n-6)/2)$, and $L$ be integers satisfying the following conditions:
\begin{enumerate}
\item
For every $j \geq 1$, $P^{\iota}b_{j}^{3}+c_{j}^{3}$ is a prime number $\equiv 2 \bmod{3}$ and prime to $P$.

\item
$L \neq \pm 1$ and $\gcd(L, b_{j}c_{j}) = 1$ for every $j \geq 0$.
Moreover, for every prime divisor $l$ of $L$, 
we have $l \neq 2$, $l \equiv 2 \bmod{3}$, and $2 \leq v_{l}(L) < n$.

\item
$P^{\iota}b_{0}- \rad{L} \cdot c_{0} = \pm 3^{k}$ with some $k \geq 0$.
Moreover, if $P \not\equiv \pm2, \pm4 \bmod{9}$, then $k = 0$.

\item
For every prime divisor $q$ of $P$ such that $q \equiv 2 \bmod{3}$, $\gcd(q, b_{0}c_{0}) = 1$.
\footnote{
	In fact, $\gcd(q, c_{0}) = 1$ follows automatically from the condition (3).
	}

\item
If $P \not\equiv \pm1 \bmod{9}$,
then $L \equiv b_{0}\prod_{j \geq 1} b_{j}^{2} \not\equiv 0 \bmod{3}$
and $\sum_{j \geq 1} b_{j}^{-1}c_{j} \not\equiv 0 \bmod{3}$.
%

\item
For every primitive triple $(x, y, z)  \in \Z^{\oplus 3}$ satisfying $x^{3}+P^{\iota}y^{3} = Lz^{n}$,
we have $x \equiv y \equiv 0 \bmod{l}$ for some prime divisor $l$ of $L$.
\end{enumerate}
Then, the equation 
\[
	(X^{3}+P^{\iota}Y^{3})(b_{0}X^{3}+ \rad{L} \cdot c_{0}Y^{3})\prod_{j = 1}^{\frac{n-6}{2}}(b_{j}^{2}X^{2}+b_{j}c_{j}XY+c_{j}^{2}Y^{2}) = LZ^{n}
\]
violates the local-global principle.
\end{proposition}

\begin{lemma} [local solubility] \label{local_even}
Let $n$ be an even integer such that $n \geq 8$.
Let $b_{j}, c_{j}, L $ $(0 \leq j \leq (n-6)/2)$ be integers satisfying the following conditions:
\begin{enumerate}
\item
For every prime divisor $q$ of $P$ such that $q \equiv 2 \bmod{3}$,
$\gcd(q, b_{0}c_{0}) = 1$.
%

\item
If $P \not\equiv \pm1 \bmod{9}$,
then $L \equiv b_{0} \prod_{j \geq 1} b_{j}^{2} \not\equiv 0 \bmod{3}$ and $\sum_{j \geq 1} b_{j}^{-1}c_{j} \not\equiv 0 \bmod{3}$.
\footnote{
	Obviously, another condition $b_{0}^{2}c_{0} \equiv \pm 1 \bmod{9}$ (or more generally $c_{0}/b_{0} \in \Q_{3}^{\times 3}$)
	is sufficient to the 3-adic solubility.
	However, this condition does not fit to our proof of \cref{main_even}.
	}
\end{enumerate}
Then, the equation 
\[
	F(X, Y, Z)
	:= (X^{3}+P^{\iota}Y^{3})(b_{0}X^{3}+c_{0}Y^{3})\prod_{j = 1}^{\frac{n-6}{2}}(b_{j}^{2}X^{2}+b_{j}c_{j}XY+c_{j}^{2}Y^{2}) - LZ^{n}
	= 0
\]
has non-trivial solutions over $\R$ and $\Q_{l}$ for every prime number $l$.
\end{lemma}

\begin{proof}
We prove this statement along Fujiwara's argument in \cite{Fujiwara}.
It is sufficient to consider the case $b_{1}c_{1} \neq 0$,
because we can find a rational point $[1 : 0 : 0]$ for $b_{1} = 0$ and $[0 : 1 : 0]$ for $c_{1} = 0$.
Then, since $\iota = 1 \ \text{or} \ 2$, and $b_{1}, c_{1} \neq 0$,
the minimal splitting field of
$(X^{3}+P^{\iota}Y^{3})(b_{1}^{2}X^{2}+b_{1}c_{1}XY+c_{1}^{2}Y^{2})$
is a Galois extension over $\Q$
whose Galois group is isomorphic to the symmetric group of degree 3.
In particular, the residual degree at every prime number is 1, 2, or 3.
Moreover, since this extension is unramified at every prime number prime to $3P$ (and $\infty$),
the polynomial $(X^{3}+P^{\iota}Y^{3})(b_{1}^{2}X^{2}+b_{1}c_{1}XY+c_{1}^{2}Y^{2})$
has a linear factor over $\R$ and over $\Q_{v}$ for every prime number $v$
such that $\gcd(v, 3P) = 1$ or $v \equiv 1 \bmod{3}$.
Here, note that
the assumption (1) ensures that
we obtain a rational point $[X : Y : Z] = [1 : -(c_{0}/b_{0})^{1/3} : 0]$ over $\Q_{q}$
for any prime divisor $q$ of $P$ such that $q \equiv 2 \bmod{3}$, and
the assumption (2) ensures that
we obtain a $3$-adic lift of $[X : Y : Z] = [1 : 0 : 1]$ (cf.\  \cite[Lemma 2.2]{Hirakawa-Shimizu_arXiv}).
\end{proof}

\begin{lemma} [global unsolubility] \label{global_even}
Let $n$ be an even integer such that $n \geq 8$.
Let $a, b_{j}, c_{j}, L$ $(0 \leq j \leq (n-6)/2)$ be integers satisfying the following conditions:
\begin{enumerate}
\item
For every $j \geq 1$, $\gcd(ab_{j}, c_{j}) = 1$ and each prime divisor $q$ of $ab_{j}^{3}+c_{j}^{3}$ satisfies $q \equiv 2 \bmod{3}$.

\item
$L \neq \pm1$ and $\gcd(L, b_{j}c_{j}) = 1$ for every $j \geq 0$.
Moreover, for every prime divisor $l$ of $L$, 
$l \equiv 2 \bmod{3}$ and $2 \leq v_{l}(L) < n$.

\item
$ab_{0}- \rad{L} \cdot c_{0} = \pm 3^{k}$ with some $k \geq 0$.
Moreover, if $a \not\equiv \pm 2, \pm4 \bmod{9}$, then $k = 0$.

\item
For every primitive triple $(x, y, z)  \in \Z^{\oplus 3}$ satisfying $x^{3}+P^{\iota}y^{3} = Lz^{n}$,
we have $x \equiv y \equiv 0 \bmod{l}$ for some prime divisor $l$ of $L$.
\end{enumerate}
Then, there exist no triples $(X, Y, Z) \in \Z^{\oplus 3} \setminus \{ (0, 0, 0) \}$ satisfying
\begin{equation} \label{condition}
	(X^{3}+aY^{3})(b_{0}X^{3}+ \rad{L} \cdot c_{0}Y^{3}) \prod_{j = 1}^{\frac{n-6}{2}} (b_{j}^{2}X^{2}+b_{j}c_{j}XY+c_{j}^{2}Y^{2}) = LZ^{n}.
\end{equation}
\end{lemma}

\begin{proof}
We prove the assertion by contradiction.
Let $(X, Y, Z) \in \Z^{\oplus 3}$ be a triple satisfying \cref{condition}.
We may assume that $\gcd(X, Y, Z) = 1$.
It is sufficient to prove that
\begin{equation} \label{gcd_even}
	\gcd((X^{3}+aY^{3})L, (b_{0}X^{3}+\rad{L} \cdot c_{0}Y^{3})(b_{j}^{2}X^{2}+b_{j}c_{j}XY+c_{j}^{2}Y^{2})) = 1 \tag{$\star$}
	\quad \text{for every $j \geq 1$}.
\end{equation}
Indeed, if \cref{gcd_even} holds,
then we have some divisor $z$ of $Z$ satisfying $X^{3}+aY^{3}= Lz^{n}$.
Hence, by the assumption, we have $X \equiv Y \equiv 0 \bmod{l}$ for some prime divisor $l$ of $L$.
However, since $v_{l}(L) < n$, we also have $Z \equiv 0 \bmod{l}$,
which contradicts that $\gcd(X, Y, Z) = 1$.
In what follows, we prove \cref{gcd_even} by contradiction.

First, suppose that
a prime divisor $q$ of $X^{3}+aY^{3}$ divides $b_{j}^{2}X^{2}+b_{j}c_{j}XY+c_{j}^{2}Y^{2}$ for some $j \geq 1$.
Then, since $\gcd(X, Y, Z) = 1$ and $v_{q}(L) < n$, we see that $Y \not\equiv 0 \bmod{q}$.
On the other hand, since $q$ divides
\[
	b_{j}^{3}(X^{3}+aY^{3}) - (b_{j}X-c_{j}Y)(b_{j}^{2}X^{2}+b_{j}c_{j}XY+c_{j}^{2}Y^{2})
	= (ab_{j}^{3}+c_{j}^{3})Y^{3},
\]
we have $ab_{j}^{3}+c_{j}^{3} \equiv 0 \bmod{q}$, hence $q \equiv 2 \bmod{3}$.
In particular, the polynomial $b_{j}^{2}T^{2}+b_{j}c_{j}T+c_{j}^{2}$ is irreducible in $\Z_{q}[T]$.
Moreover, since $b_{j}^{2}X^{2}+b_{j}c_{j}XY+c_{j}^{2}Y^{2} \equiv 0 \bmod{q}$
and $Y \not\equiv 0 \bmod{q}$,
we have $c_{j} \equiv 0 \bmod{q}$, hence $ab_{j} \equiv 0 \bmod{q}$.
However, it contradicts the assumption that $\gcd(ab_{j}, c_{j}) = 1$.

Secondly, suppose that a prime divisor $l$ of $L$ divides
$b_{j}^{2}X^{2}+b_{j}c_{j}XY+c_{j}^{2}Y^{2}$ for some $j \geq 1$.
Then, since $l \equiv 2 \bmod{3}$ and $\gcd(L, b_{j}c_{j}) = 1$,
we have $X \equiv Y \equiv 0 \bmod{l}$.
However, since $v_{l}(L) < n$,
we see that $Z \equiv 0 \bmod{l}$,
which contradicts that $\gcd(X, Y, Z) = 1$.

Thirdly, suppose that a prime divisor $q$ of $X^{3}+aY^{3}$ divides $b_{0}X^{3}+ \rad{L} \cdot c_{0}Y^{3}$.
Then, since $\gcd(X, Y, Z) = 1$ and $v_{q}(L) < n$, we see that $Y \not\equiv 0 \bmod{q}$.
On the other hand, since $q$ divides
\[
	b_{0}(X^{3}+aY^{3}) - (b_{0}X^{3}+\rad L \cdot c_{0}Y^{3})
	= (ab_{0}-\rad L \cdot c_{0})Y^{3} = \pm3^{k}Y^{3},
\]
we have $q = 3$ and $k \geq 1$, hence $a \equiv \pm2, \pm4 \bmod{9}$.
However, $X^{3}+aY^{3} \equiv 0 \bmod{3}$ implies that $X \equiv aY \equiv 0 \bmod{3}$,
which contradicts that $Y \not\equiv 0 \bmod{3}$ and $a \equiv \pm2, \pm4 \bmod{9}$.

Finally, suppose that a prime divisor $l$ of $L$ divides $b_{0}X^{3}+ \rad{L} \cdot c_{0}Y^{3}$.
Then, since $\gcd(L, b_{0}) = 1$, we have $X \equiv 0 \bmod{l}$.
Since $v_{l}(L) \geq 2$, we have $(a \cdot \rad{L} \cdot c_{0}\prod_{j \geq 1} c_{j}^{2})Y^{n} \equiv 0 \bmod{l^{2}}$.
On the other hand, since $ab_{0}-\rad{L} \cdot c_{0} = \pm 3^{k}$, we have $\gcd(l, a) = 1$.
Moreover, since $\gcd(L, c_{j}) = 1$ for every $j \geq 0$,
we have $\gcd(l, ac_{0}\prod_{j \geq 1} c_{j}^{2}) = 1$ for every $j \geq 0$, hence $Y \equiv 0 \bmod{l}$.
However, since $v_{l}(L) < n$, we have $Z \equiv 0 \bmod{l}$,
which contradicts that $\gcd(X, Y, Z) = 1$.
This completes the proof.
\end{proof}

\section{Reduction to the Fermat type equations $X^{3}+P^{\iota}Y^{3} = l^{m}Z^{n}$}

Let $p$ be a prime number and
$u$ be a square-free integer prime to $p$.
Set $P = pu$.
Let $\pi = P^{1/3} \in \R$ be the real cubic root of $P$,
$K = \Q(\pi)$ be the pure cubic field generated by $\pi$,
$\mathcal{O}_{K}$ be the ring of integers in $K$,
and $\epsilon = \alpha+\beta\pi+\gamma\pi^{2} > 1$ be the fundamental unit of $K$
with $\alpha, \beta, \gamma \in (1/3)\Z$.
Note that the Galois closure of $K$ in the field $\C$ of complex numbers is $K(\zeta_{3})$,
where $\zeta_{3} \in \C$ is a fixed primitive cubic root of unity.
For basic properties of these objects, see e.g. \cite{Dedekind}.

In what follows, we assume that $P \not\equiv \pm1 \bmod{9}$.
Then, since we assume that $u$ is square-free, we have $\mathcal{O}_{K} = \Z[\pi]$ and so $\alpha, \beta, \gamma \in \Z$.

\begin{theorem} \label{Fermat_unify}
In the above setting, further suppose that
$\beta \equiv 0 \bmod{p}$ and the class number of $K$ is prime to $p$.
Let $n$ and $m$ be positive integers such that $n \geq 3$, $n \equiv 0 \bmod{p}$, and $m \not\equiv 0 \bmod{p}$.
Then, there exist infinitely many odd prime numbers $l$ such that
$l \equiv 2 \bmod{3}$ and for every primitive triple $(x, y, z) \in \Z^{\oplus 3}$ satisfying $x^{3}+P^{2}y^{3} = l^{m}z^{n}$,
we have $x \equiv y \equiv 0 \bmod{l}$.
\end{theorem}

We will prove \cref{Fermat_unify} in the next section.
Here, we prove \cref{main_even,main_unify} by using \cref{Fermat_unify}.
We start from the former.

\begin{proof} [Proof of \cref{main_even} under \cref{Fermat_unify}]
We prove the assertion by taking desired parameters so that the conditions in \cref{recipe_even} hold.
Note that since every step is constructive,
we obtain an algorithm to produce arbitrarily many non-singular plane curves violating the local-global principle
up to generation of the prime numbers $l$ claimed in \cref{Fermat_unify}.
An algorithm to produce these prime numbers $l$ is explained in the top of the next section.

Let $u$ be an odd square-free integer such that
$P = 2u \not\equiv \pm 1 \bmod{9}$, $\beta \equiv 0 \bmod{2}$, and the class number of $K = \Q(P^{1/3})$ is odd.
First, we consider the case $u \not\equiv 0 \bmod{3}$, i.e., $P \equiv \pm 2, \pm4 \bmod{9}$.

Let $f(X, Y) = P^{2}(3X+1)^{3} + (3Y+1)^{3}$ and $g(X, Y) = P^{2}(3X-1)^{3} + (3Y)^{3}$.
Then, since $\gcd(f(0, 0), f(0, -1)) = 1$ and $\gcd(g(0, 0), g(0, 1)) = 1$,
we have $\gcd(f(\Z^{\oplus 2})) = 1$ and $\gcd(g(\Z^{\oplus 2})) = 1$ respectively.
By \cref{HBM},
there exist infinitely many distinct prime numbers of the form
$q = f(B, C)$ (resp.\ $q = g(B, C)$) with $(B, C) \in \Z^{\oplus 2}$ and prime to $P$.
Note that all of them satisfy $q \equiv 2 \bmod{3}$.
Among such prime numbers $q$,
take distinct $(n-6)/2$ prime numbers $q_{j} = f(B_{j}, C_{j}) \ \text{or} \ g(B_{j}, C_{j})$
with $(B_{j}, C_{j}) \in \Z^{\oplus 2}$ ($1 \leq j \leq (n-6)/2$)
so that $\gcd(3, \sum_{j} b_{j}^{-1}c_{j}) = 1$,
where $(b_{j}, c_{j}) := (3B_{j}+1, 3C_{j}+1) \ \text{or} \ (3B_{j}-1, 3C_{j})$
according to whether $q_{j} = f(B_{j}, C_{j}) \ \text{or} \ g(B_{j}, C_{j})$.

For each $(n-6)/2$-tuple $((b_{j}, c_{j}))_{1 \leq j \leq (n-6)/2}$ taken as above,
we have infinitely many prime numbers $l$ satisfying the properties claimed in \cref{Fermat_unify}
with an additional condition that $\gcd(l, Pb_{j}c_{j}) = 1$ for every $j \geq 1$.
We fix such $l$ arbitrarily.

Let $Q$ be the product of the prime divisors $q$ of $P$ such that $q \equiv 2 \bmod{3}$.
Then, since $\gcd(l, 3P) = 1$,
there exist infinitely many pairs $(b, c_{0}) \in \Z^{\oplus 2}$
such that $\gcd(l, c_{0}) = 1$ and
\[
	3P^{2}Qb-lc_{0} = \pm3^{k} + P^{2}, \quad \text{i.e.,} \quad P^{2}(-1+3Qb)-lc_{0} = \pm3^{k}
\]
with some $k \geq 0$.
Take such a pair $(b, c_{0})$ arbitrarily and set $b_{0} = -1+3Qb$.
Then, since $l \equiv 2 \bmod{3}$ and $m$ is odd,
we have $b_{0} \prod_{j \geq 1} b_{j}^{2} \equiv 2 \equiv l^{m} \bmod{3}$ and $\gcd(Q, b_{0}) = 1$.
Moreover, we see that $l \neq 2$ and $\gcd(l, b_{0}c_{0}) = 1$.
Therefore, \cref{recipe_even} implies that the equation
\[
	(X^{3}+P^{2}Y^{3})(b_{0}X^{3}+lc_{0}Y^{3}) \prod_{j = 1}^{\frac{n-6}{2}}(b_{j}^{2}X^{2} + b_{j}c_{j}XY + c_{j}^{2}Y^{2}) = l^{m} Z^{n}
\]
violates the local-global principle.

The non-singularity is a consequence of the following two facts:
\begin{enumerate}
\item
Since $q_{j}$ and $q_{k}$ are distinct prime numbers,
$[b_{j} : c_{j}] \neq [b_{k} : c_{k}]$ for any distinct $j, k \geq 1$.

\item
Since $\gcd(l, b_{0}c_{0}) = 1$ and all of $l, b_{0}, c_{0}$ are odd,
two polynomials $X^{3}+P^{2}Y^{3}$ and $b_{0}X^{3}+lc_{0}Y^{3}$ are both irreducible in $\Q[X, Y]$
and cannot have a common root in $\C$.
\end{enumerate}

The infinitude of the geometric isomorphy classes of plane curves follows from Schwarz's theorem on the finiteness of the automorphism group of a non-singular algebraic curve of genus $\geq 2$.
For the detail, see \cite[Lemma 4.1]{Hirakawa-Shimizu_arXiv}.
This completes the proof in the case $u \not\equiv 0 \bmod{3}$.

Finally, if $u \equiv 0 \bmod{3}$,
then we take $f(X, Y) = P^{2}(3X+1)^{3} + (3Y+1)^{3}$ and $g(X, Y) = P^{2}(3X-1)^{3} + (3Y+1)^{3}$
so that we can take $(b_{j}, c_{j})$ such that $\sum_{j \geq 1} b_{j}^{-1}c_{j} \not\equiv 0 \bmod{3}$.
In this case, $b_{0} = -1+3Qb$ and $c_{0} \in \Z$ with $P^{2}b_{0} - lc_{0} = \pm1$ works well.
This completes the proof.
\end{proof}

Next, we prove \cref{main_unify} by using \cref{Fermat_unify}.
Here, recall the following proposition,
which is a counterpart of \cref{recipe_even}
and reduces \cref{main_unify}
to the generation of prime numbers of the form $X^{3}+P^{2}Y^{3}$
and a certain Fermat type equation $x^{3}+P^{2}y^{3} = Lz^{n}$.

\begin{proposition} [{cf.\ \cite[Proposition 2.1]{Hirakawa-Shimizu_arXiv}}] \label{recipe_odd}
Let $n$ be an odd integer such that $n \geq 5$,
$p$ be a prime number,
and $P = p \ \text{or} \ 2p$,
$\iota = 1 \ \text{or} \ 2$.
Let $b_{j}, c_{j}, L$ $(1 \leq j \leq (n-3)/2)$ be integers satisfying the following conditions:
\begin{enumerate}
\item
For every $j$, $P^{\iota}b_{j}^{3}+c_{j}^{3}$ is a prime number $\equiv 2 \bmod{3}$ and prime to $P$.

\item
$\gcd(L, b_{j}c_{j}) = 1$ for every $j$.
Moreover, for every prime divisor $l$ of $L$, $l \equiv 2 \bmod{3}$ and $v_{l}(L) < n$.

\item
For every prime divisor $q$ of $P$ such that $q \equiv 2 \bmod{3}$ (only $q = 2, p$ are possible),
we have $L \equiv \prod_{j} b_{j}^{2} \not\equiv 0 \bmod{q}$ and $\sum_{j} b_{j}^{-1}c_{j} \not\equiv 0 \bmod{q}$.

\item
If $P \not\equiv \pm1 \bmod{9}$,
then $L \equiv \prod_{j} b_{j}^{2} \not\equiv 0 \bmod{3}$
and $\sum_{j} b_{j}^{-1}c_{j} \not\equiv 0 \bmod{3}$.

\item
For every primitive triple $(x, y, z)  \in \Z^{\oplus 3}$ satisfying $x^{3}+P^{\iota}y^{3} = Lz^{n}$,
there exists a prime divisor $l$ of $L$ such that $x \equiv y \equiv 0 \bmod{l}$.
\end{enumerate}
Then, the equation 
\[
	(X^{3}+P^{\iota}Y^{3})\prod_{j = 1}^{\frac{n-3}{2}}(b_{j}^{2}X^{2}+b_{j}c_{j}XY+c_{j}^{2}Y^{2}) = LZ^{n}
\]
violates the local-global principle.
\end{proposition}

\begin{proof} [Proof of \cref{main_unify} under \cref{Fermat_unify}]
The proof is almost parallel to the proof of \cref{main_even}.
We prove the assertion by taking desired parameters so that the conditions in \cref{recipe_odd} hold.
In this time, we do not need to take extra parameters $b_{0}$ and $c_{0}$.

Let $p$ be an odd prime number.
Take $u = 2 \ \text{or} \ 1$, i.e., $P = 2p \ \text{or} \ p$ so that $P \not\equiv \pm1 \bmod{9}$.
Thanks to \cref{Hirakawa-Shimizu} (see also \cref{outline}),
it is sufficient to consider the case where $p \neq 3$ and $\beta \equiv 0 \bmod{p}$.
Note that since the class numbers of $\Q(p^{1/3})$ and $\Q((2p)^{1/3})$ are prime to $p$ (cf.\ \cite[Lemma 3.4]{Hirakawa-Shimizu_arXiv}),
we can actually apply \cref{Fermat_unify}.


Let $f(X, Y) = P^{2}(3PX+1)^{3} + (3PY+1)^{3}$ and $g(X, Y) = P^{2}(3PX-1)^{3} + (3PY+3)^{3}$.
Then, since $\gcd(f(0, 0), f(0, 1), f(0, -1)) = 1$ and $\gcd(g(0, 0), g(0, 1), g(0, -1)) = 1$,
we have $\gcd(f(\Z^{\oplus 2})) = 1$ and $\gcd(g(\Z^{\oplus 2})) = 1$ respectively.
By \cref{HBM},
there exist infinitely many distinct prime numbers of the form
$q = f(B, C)$ (resp.\ $q = g(B, C)$) with $(B, C) \in \Z^{\oplus 2}$ and prime to $P$.
Note that all of them satisfy $q \equiv 2 \bmod{3}$.
Among such prime numbers $q$, take distinct $(n-3)/2$ prime numbers $q_{j} = f(B_{j}, C_{j}) \ \text{or} \ g(B_{j}, C_{j})$
with $(B_{j}, C_{j}) \in \Z^{\oplus 2}$ ($1 \leq j \leq (n-6)/2$) so that $\gcd(3p, \sum_{j} b_{j}^{-1}c_{j}) = 1$,
where $(b_{j}, c_{j}) := (3PB_{j}+1, 3PC_{j}+1) \ \text{or} \ (3PB_{j}-1, 3PC_{j}+3)$
according to whether $q_{j} = f(B_{j}, C_{j}) \ \text{or} \ g(B_{j}, C_{j})$.
Here, note that since $n$ is odd,
$\sum_{j} b_{j}^{-1}c_{j}$ is odd whenever $P$ is even.

For  each $(n-3)/2$-tuple $((b_{j}, c_{j}))_{1 \leq j \leq (n-3)/2}$ taken as above,
we have infinitely many prime numbers $l$ satisfying the properties claimed in \cref{Fermat_unify}
with an additional condition that $\gcd(l, 2Pb_{j}c_{j}) = 1$ for every $j$.
We fix such $l$ arbitrarily.
Then, since both $l$ and $p$ are odd and $b_{j} \equiv \pm1 \bmod{3P}$ for every $j$,
we see that $l^{p-1} \equiv \prod_{j \geq 1} b_{j}^{2} \equiv 1 \bmod{3P}$.
Therefore, \cref{recipe_odd} implies that the equation
\[
	(X^{3}+P^{2}Y^{3}) \prod_{j = 1}^{\frac{n-6}{2}}(b_{j}^{2}X^{2} + b_{j}c_{j}XY + c_{j}^{2}Y^{2}) = l^{p-1} Z^{n}
\]
violates the local-global principle.

The non-singularity is a consequence of the following two facts:
\begin{enumerate}
\item
Since $q_{j}$ and $q_{k}$ are distinct prime numbers,
$[b_{j} : c_{j}] \neq [b_{k} : c_{k}]$ for any distinct $j, k$.

\item
Since $X^{3}+P^{2}Y^{3}$ is irreducible in $\Q[X, Y]$, it cannot have a multiple root nor a common root with $b_{j}^{2}X^{2} + b_{j}c_{j}XY + c_{j}^{2}Y^{2}$ for any $j$.
\end{enumerate}

The infinitude of the geometric isomorphy classes of plane curves follows from Schwarz's theorem on the finiteness of the automorphism group of a non-singular algebraic curve of genus $\geq 2$.
For the detail, see \cite[Lemma 4.1]{Hirakawa-Shimizu_arXiv}.
This completes the proof.
\end{proof}

\section{Proof of \cref{Fermat_unify}}

In order to prove \cref{Fermat_unify}, we again use \cref{HBM}.
\footnote{
	Here, the use of \cref{HBM} is not essential
	unlike the situation in the proof of \cref{main_even} (resp.\ \cref{main_unify}) under \cref{Fermat_unify}.
	In fact, the constraint for $(a, b, c) \in \Z^{\oplus 3}$ in \cref{Fujiwara_Fermat_unify} is relatively weak.
	In particular, it is sufficient for the proof of \cref{Fermat_unify}
	that there exist infinitely many prime numbers of the form
	$l = N_{K/\Q}(a+b\pi+c\pi^{2})$ with $b \not\equiv 0 \bmod{p}$,
	which we can verify by the ring class field theory and the Chebotarev density theorem
	(cf.\ \cite[Ch.\ V]{Janusz_text} and \cite{Lv-Deng}).
	}
Let $p$ be a prime number satisfying the conditions in \cref{Fermat_unify}.
Let  $h(A, B) = (3A-1)^{3}+P(3B+1)^{3}$ or $h(A, B) = (3pA-1)^{3}+P(3pB+3)^{3}$ according to $p = 3$ or not.
\footnote{
	If $p = 2$ and $u$ has no prime divisors of the form $q \equiv 2 \bmod{3}$,
	then $h(A, B) = (6A+1)^{3}+P(6B-1)^{3}$ also works.
	More generally, if we want to obtain concrete examples with small coefficients,
	we can easily modify the polynomials generating the prime numbers with desired properties.
	}
Then, since $\gcd(h(0, 0), h(1, 0), h(-1, 0)) = 1$,
we have $\gcd(h(\Z^{\oplus 2})) = 1$.
Therefore, \cref{HBM} implies that
there exist infinitely many prime numbers $l$ of the form
\[
	l = a^{3}+Pb^{3} \equiv 2 \bmod{3}
	\quad \text{with} \quad
	(a, b) = \begin{cases}
	(3A-1, 3B+1) & \text{if $p = 3$} \\
	(3pA-1, 3pB+3) & \text{if $p \neq 3$}
	\end{cases}
\]
Thus, \cref{Fermat_unify} is obtained from the case of $(\iota, \nu) = (2, 1)$ in the following proposition.

\begin{proposition} \label{Fujiwara_Fermat_unify}
Let $(\iota, \nu) = (1, 2), (2, 1)$.
Let $p$ be a prime number and $u$ be a square-free positive integer prime to $p$.
Set $P = pu$ and $\pi = P^{1/3} \in\R$.
Suppose that $P \not\equiv \pm 1 \bmod{9}$,
the fundamental unit $\epsilon = \alpha+\beta\pi+\gamma\pi^{2}$ of $K = \Q(\pi)$
with $\alpha, \beta, \gamma \in \Z$ satisfies $\beta \equiv 0 \bmod{p}$,
and the class number of $K$ is prime to $p$.
If $(\iota, \nu) = (1, 2)$, then further assume that $\gamma \equiv 0 \bmod{p}$.
Let $l$ be a prime number such that $\gcd(l, P) = 1$ and $l \equiv 2 \bmod{3}$.
Assume that there exist $a+b\pi+c\pi^{2} \in \mathcal{O}_{K}$
with $a, b, c \in \Z$ and $m \in \Z_{\geq 1}$ satisfying the following conditions:
\begin{enumerate}
\item
	$l = N_{K/\Q}(a+b\pi+c\pi^{2}) \ (= a^{3}+b^{3}P+c^{3}P^{2}-3abcP)$.

\item
	\begin{enumerate}
	\item
	If $(\iota, \nu) = (1, 2)$, then $\binom{m}{2}b^{2}+mac$ is prime to $p$.
	\footnote{In fact, the condition $\nu = 2$ is too strict for odd $p$,
	and $p^{\nu} \geq 3$ is sufficient.}
	
	\item
	If $(\iota, \nu) = (2, 1)$, then $mb$ is prime to $p$.
	\end{enumerate}
\end{enumerate}
Then, for every integer $n \geq 3$ divisible by $p^{\nu}$
and every primitive triple $(x, y, z)  \in \Z^{\oplus 3}$ satisfying $x^{3}+P^{\iota}y^{3} = l^{m}z^{n}$,
we have $x \equiv y \equiv 0 \bmod{l}$.
\end{proposition}

\begin{proof}
The proof is almost parallel to the proof of \cite[Theorem 3.1]{Hirakawa-Shimizu_arXiv} (cf.\ \cref{release}).
We prove the assertion by contradiction.
Suppose that there exists a primitive triple $(x, y, z) \in \Z^{\oplus 3}$
such that $x^{3}+P^{\iota}y^{3} = l^{m}z^{n}$, and either $x$ or $y$ is prime to $l$.

First, note that
since either $x$ or $y$ is prime to $l$ and $\gcd(l, P) = 1$,
$x^{2}-xy\pi^{\iota}+y^{2}\pi^{2\iota}$ cannot be divisible by $l$.
Moreover,
$l \equiv 2 \bmod{3}$ splits in $K$ to the product of two prime ideals
$\mathfrak{p}_{l}$ and $\mathfrak{p}_{l^{2}} $ of norms $l$ and $l^{2}$ respectively.
Suppose that $x+y\pi$ is divisible by $\mathfrak{p}_{l^{2}}$.
Then, the product of its conjugates
$(x+\zeta_{3} y\pi^{\iota})(x+\zeta_{3}^{2} y\pi^{\iota})
= x^{2}-xy\pi^{\iota}+y^{2}\pi^{2\iota}$
is divisible by $l$, a contradiction
(cf.\ the following argument for $q \equiv 2 \bmod{3}$).
Therefore, $x^{2}-xy\pi^{\iota}+y^{2}\pi^{2\iota}$ is divisible by $\mathfrak{p}_{l^{2}}^{m}$
but not divisible by $\mathfrak{p}_{l}$.
Accordingly,
$x+y\pi^{\iota}$ is divisible by $\mathfrak{p}_{l}^{m}$ but not divisible by $\mathfrak{p}_{l^{2}}$.

Next, suppose that $x+y\pi^{\iota}$ is divisible by a prime ideal above a prime divisor $q$ of $z$.
Then, since $P$ is square-free, $\iota < 3$, and $(x, y, z)$ is primitive,
we see that $\gcd(q, P) = 1$.
Moreover,
both of $x+y\pi^{\iota}$ and $x^{2}-xy\pi^{\iota}+y^{2}\pi^{2\iota}$ are not divisible by $q$ itself
because if either $x^{3}+P^{\iota}y^{3}$ is divisible by $q$ itself,
then we have $x \equiv y \equiv 0 \bmod{q}$, which contradicts that $(x, y, z)$ is primitive.
On the other hand, since $P \not\equiv \pm1 \bmod{9}$,
the possible decomposition types of $q$ in $K$ are as follows:
\begin{enumerate}
\item
	$(q) = \mathfrak{p}_{q, 1} \mathfrak{p}_{q, 2} \mathfrak{p}_{q, 3}$,
	i.e., $q \equiv 1 \bmod{3}$ and $P \bmod{q} \in \mathbb{F}_{q}^{\times 3}$.
	
\item
	$(q) = \mathfrak{p}_{q} \mathfrak{p}_{q^{2}}$, i.e., $q \equiv 2 \bmod{3}$ and $\gcd(q, P) = 1$.

\item
	$(q) = \mathfrak{p}_{q}^{3}$, i.e., $q = 3$.
\end{enumerate}
In each case, we have the following conclusion:
\begin{enumerate}
\item
	If $x+y\pi^{\iota}$ is divisible by
	distinct two prime ideals above $q$, say $\mathfrak{p}_{q, 1}$ and $\mathfrak{p}_{q, 2}$, 
	then $x^{2}-xy\pi^{\iota}+y^{2}\pi^{2\iota}$
	is divisible by $(\mathfrak{p}_{q, 1}\mathfrak{p}_{q, 3}) \cdot (\mathfrak{p}_{q, 2}\mathfrak{p}_{q, 3})$,
	hence by $q$, a contradiction.
	Therefore, we may assume that $x+y\pi^{\iota}$ is divisible by $\mathfrak{p}_{q, 1}^{nv_{q}(z)}$
	but not by $\mathfrak{p}_{q, 2}$ nor $\mathfrak{p}_{q, 3}$
	by replacing $\mathfrak{p}_{q, 1}, \mathfrak{p}_{q, 2}, \mathfrak{p}_{q, 3}$ to each other if necessary.
	
\item
	In this case,
	$q$ is decomposed in $K(\zeta_{3})$ so that
	$\mathfrak{p}_{q} = \mathfrak{P}_{q^{2}, 1}$ and
	$\mathfrak{p}_{q^{2}} = \mathfrak{P}_{q^{2}, 2}\mathfrak{P}_{q^{2}, 3}$.
	If $x+y\pi^{\iota}$ is divisible by $\mathfrak{p}_{q^{2}}$,
	then	$x^{2}-xy\pi^{\iota}+y^{2}\pi^{2\iota}$ is divisible by
	$(\mathfrak{P}_{q^{2}, 1}\mathfrak{P}_{q^{2}, 2}) \cdot (\mathfrak{P}_{q^{2}, 1}\mathfrak{P}_{q^{2}, 3})$,
	hence by $q$, a contradiction.
	Therefore, $x+y\pi^{\iota}$ is divisible by $\mathfrak{p}_{q}^{nv_{q}(z)}$
	but not by $\mathfrak{p}_{q^{2}}$.

\item
	In this case, since $x^{3}+P^{\iota}y^{3}$ is divisible by $\mathfrak{p}_{3}^{3n}$,
	$x+y\pi^{\iota}$ is divisible by $\mathfrak{p}_{3}^{n}$.
	Since $n \geq 3$, $x+\pi^{\iota}y$ is divisible by $3$.
	Moreover, since $P$ is square-free and $\iota < 3$, $\pi^{\iota}$ cannot be divisible by $3$.
	Hence, both $x$ and $y$ are divisible by $3$,
	which contradicts that $(x, y, z)$ is primitive.
\end{enumerate}

As a consequence, we see that
there exists an integral ideal $\mathfrak{w}$ of $\mathcal{O}_{K}$ such that
\[
	(x+y\pi^{\iota})
	= \mathfrak{p}_{l}^{m}\mathfrak{w}^{n} \quad
	\text{and} \quad
	(\mathfrak{w}, 3P) = 1.
\]
Then, since the first assumption implies that $\mathfrak{p}_{l}$ is generated by $a+b\pi+c\pi^{2}$,
$\mathfrak{w}^{n}$ is a principal ideal.
Moreover, since we assume that the class number of $K$ is prime to $p$,
the ideal $\mathfrak{w}^{n/p^{\nu}}$ itself is also generated by a single element
$w_{0}+w_{1}\pi+w_{2}\pi^{2} \in \mathcal{O}_{K}$ with $w_{0}, w_{1}, w_{2} \in \Z$.
Therefore, there exists $k \in \Z$ such that
\[
	x+y\pi^{\iota}
	= \epsilon^{k}(a+b\pi+c\pi^{2})^{m}(w_{0}+w_{1}\pi+w_{2}\pi^{2})^{p^{\nu}}.
\]
Since we assume that $\beta \equiv 0 \bmod{p}$,
we have
\[
	x+y\pi^{\iota}
	\equiv (\alpha+k\gamma\pi^{2})a^{m-2}\left( a^{2}+mab\pi+\left( \binom{m}{2}b^{2}+mac \right)\pi^{2} \right)(w_{0}+w_{1}\pi^{p^{\nu}}) \bmod{p}.
\]
Here, note that since $\gcd(l, P) = 1$, $a \not\equiv 0 \bmod{p}$.
Furthermore, since $\gcd(\mathfrak{w}, P) = 1$, $w_{0} \not\equiv 0 \bmod{p}$.

If $(\iota, \nu) = (1, 2)$,
then since we assume that $\gamma \equiv 0 \bmod{p}$,
the above congruence between the coefficients of $\pi^{2}$
contradicts the assumption.

If $(\iota, \nu) = (2, 1)$,
then the above congruence between the coefficients of $\pi$
contradicts the assumption.
This completes the proof.
\end{proof}

\begin{remark} \label{release}
It should be noted that a key ingredient in the previous work \cite[Theorem 3.1]{Hirakawa-Shimizu_arXiv} is a counterpart of \cref{Fujiwara_Fermat_unify} in the cases $(\iota, \nu) = (1, 1), (2, 2)$.
This restriction on $(\iota, \nu)$ binds \cref{Hirakawa-Shimizu} with the assumption that the degree $n$ is divisible by $p^{\iota}$.
In this time, we are released from this strong assumption thanks to \cref{Fujiwara_Fermat_unify}.
\end{remark}

\begin{remark}[Outline of the proof \cref{Hirakawa-Shimizu} in the case $\beta \not\equiv 0 \bmod{p}$] \label{outline}
For the completeness,
we explain an outline of the proof \cref{Hirakawa-Shimizu} in the case $\beta \not\equiv 0 \bmod{p}$,
which was essentially given in \cite{Hirakawa-Shimizu_arXiv}.
The idea is similar to the case $\beta \equiv 0 \bmod{p}$ as given in the previous section.
In the case $\beta \not\equiv 0 \bmod{p}$,
we produce the parameters $(b_{j}, c_{j})$ $(1 \leq j \leq (n-3)/2)$ by using appropriate polynomials e.g.
$f(X, Y) = P(3PX\pm1)^{3} + (3PY+1)^{3}$ and $g(X, Y) = P(3PX\mp1)^{3} + (3PY+3)^{3}$
according to $P \equiv \pm 1 \bmod{3}$.
The crucial point appears in the difference between the proofs of \cite[Theorem 3.1]{Hirakawa-Shimizu_arXiv} and \cref{Fujiwara_Fermat_unify}:
In the former,
in order to deduce a contradiction from the modulo $p$ comparison of the coefficients of $\pi^{2}$,
it is sufficient that the coefficient $l^{m} = N_{K/\Q}((a+b\pi+c\pi^{2})^{m})$ of $z^{n}$ satisfies additional technical conditions that
$b \equiv 0 \bmod{p}$ and a curious quantity
\[
	\left( \frac{\beta}{2\alpha}-\frac{\gamma}{\beta} \right)^{2} - \frac{2c}{a} \cdot m
\]
is not a quadratic residue modulo $p$ (cf.\ \cite[Lemma 3.5]{Hirakawa-Shimizu_arXiv}).
This condition is verified by generating prime numbers of the form
e.g. $l = h(A, B) = (3PA+1)^{3}+P^{2}(3PC+1)$ and considering sufficiently many even integers $m$ in the range $1 \leq m \leq p-1$ (cf.\ effective P\'olya-Vinogradov inequalities as obtained in \cite{Pomerance}).
Here, we take even $m$ in order to ensure the $3$-adic and $p$-adic solubility (cf.\ \cref{recipe_odd}).
Similarly, if $p = 3$, we can produce $(b_{j}, c_{j})$ $(1 \leq j \leq (n-3)/2)$ by e.g.
$f(X, Y) = P^{\iota}(PX+1)^{3}+(PY-1)^{3}$ and $g(X, Y) = P^{\iota}(PX-1)^{3}+(PY-1)^{3}$
and generate $l$ by e.g. $h(A, C) = (6A-1)^{3}+P^{2\iota}(6C+1)$.
\end{remark}

\begin{remark}
If $p = 2$,
then $\binom{m}{2}b^{2}+mac \equiv \binom{m}{2}b^{2}+mc \bmod{2}$,
and it is prime to $p = 2$ if and only if one of the following conditions holds:
\begin{enumerate}
\item
$m \equiv 1 \bmod{4}$ and $c$ is odd.

\item
$m \equiv 2 \bmod{4}$ and $b$ is odd.

\item
$m \equiv 3 \bmod{4}$ and $b+c$ is odd.
\end{enumerate}
\end{remark}

\begin{remark} \label{infinite_cubic_forms}
It seems plausible that
there exist infinitely many odd integers $u$ such that the class numbers of the cubic fields $\Q((2u)^{1/3})$ are odd:
In fact,
the class number of $\Q((2u)^{1/3})$ is odd
if at least one of the $2$-Selmer groups of the two elliptic curves $E^{(\pm u)}$ defined by $y^{2} = x^{3} \pm 2u$ is trivial (cf.\ \cite{Li_2-Selmer}).
Note also that the class number of $\Q((2u)^{1/3})$ is odd
if and only if the class number of its Galois closure $\Q((2u)^{1/3}, \zeta_{3})$ is odd (cf.\ \cite{Reich}).
For related topics, see e.g. \cite{Ho-Shankar-Varma,Taniguchi-Thorne}.
\end{remark}

\section{A variant of \cref{main_even}}

In \S3, we proved \cref{Fermat_unify} from the case $(\iota, \nu) = (2, 1)$ in \cref{Fujiwara_Fermat_unify} with $p = 2$.
It is obvious that
we can prove a counterpart of \cref{Fermat_unify} from the other case $(\iota, \nu) = (1, 2)$.
For instance, we obtain the following corollary
by applying the exactly same argument as the proof of \cref{Fermat_unify} with $p = 2$
(with the same polynomial $h(A, B) = (6A-1)^{3}+P(6B+3)^{3}$).

\begin{corollary} \label{Fermat_deg4}
Let $u$ be an odd square-free integer such that $u \not\equiv \pm4 \bmod{9}$.
Set $P = 2u$.
Let
$\pi = P^{1/3} \in \R$ be the real cubic root of $P$,
$K = \Q(\pi)$ be the pure cubic field generated by $\pi$,
and $\epsilon = \alpha+\beta\pi+\gamma\pi^{2}$ be the fundamental unit of $K$
with $\alpha, \beta, \gamma \in \Z$.
Suppose that both $\beta$ and $\gamma$ are even and the class number of $K$ is odd.
Then, there exist infinitely many odd prime numbers $l \equiv 2 \bmod{3}$ such that
there exist no primitive solution of $x^{3}+Py^{3} = l^{2}z^{4}$.
\end{corollary}

Moreover,
by combining \cref{Fermat_deg4} and appropriate prime generating polynomials,
we can prove a variant of \cref{main_even}.
For instance,
if $u \equiv \pm 1 \bmod{3}$,
then by taking $f(X, Y) = P(3X \mp 1)^{3} + (3Y+1)^{3}$ and $g(X, Y) = P(3X \pm 1)^{3} + (3Y)^{3}$
in place of $f$ and $g$ in the proof of \cref{main_even}
and $b_{0} = 1+3Qb$ in place of $-1+3Qb$ so that $Pb_{0} - lc_{0} = \pm 3^{k}$ with some $k \geq 0$,
we obtain the following.

\begin{theorem} \label{main_deg4}
Let $n$ be an integer such that $n \geq 8$ and $n \equiv 0 \bmod{4}$,
and $P$ be as in \cref{Fermat_deg4}.
Then, there exist
infinitely many $(n-6)/2$-tuples of pairs of integers $(b_{j}, c_{j})$ $(1 \leq j \leq (n-6)/2)$
satisfying the following condition:

There exist infinitely many prime numbers $l$ and infinitely many pairs of integers $(b_{0}, c_{0})$ such that
the equation
\begin{equation} \label{equation_deg4}
	(X^{3}+PY^{3})(b_{0}X^{3}+lc_{0}Y^{3}) \prod_{j = 1}^{(n-6)/2} (b_{j}^{2}X^{2}+b_{j}c_{j}XY+c_{j}^{2}Y^{2}) = l^{2}Z^{n}
\end{equation}
define non-singular plane curves which violate the local-global principle.

Moreover, for each $n$,
there exists a set of such $(n-4)/2$-tuples $((b_{j}, c_{j}))_{0 \leq j \leq (n-6)/2}$
which gives infinitely many geometrically non-isomorphic classes of such curves.
\end{theorem}

\section{Numerical examples of degree 8}

In this section, we demonstrate that
\cref{main_even,main_deg4} actually give explicit equations of non-singular plane curve which violate the local-global principle.
We emphasize that our construction has a character in contrast to examples obtained by Nguyen in \cite{Nguyen_QJM,Nguyen_Tokyo} (cf.\ \cref{Nguyen_QJM,Nguyen_Tokyo}).
Recall that Nguyen constructed infinitely many plane curves of even degree which violate the local-global principle
but explained by the Brauer-Manin obstruction on certain hyperelliptic curves covered by the plane curves.
Here, we give two numerical examples for \cref{main_even,main_deg4} respectively,
both of which are $\Z/4\Z$-coverings of hyperellitpic curves with $\Q$-rational points.
This means that the violation of the local-global principle for the former plane curves cannot be explained
by the Brauer-Manin obstruction on the latter hyperelliptic curves.

\subsection{Example for \cref{main_even} with $u = 7$}

First, we construct an example for \cref{main_even} in the case of $n = 8$ and $u = 7$.
As a special value of $f(X, Y) = 14^{2}(3X+1)^{3}+(3Y+1)^{3}$,
we obtain a prime number $197 = 14^{2} \cdot 1^{3} + 1^{3}$ with $(b_{1}, c_{1}) = (1, 1)$.
Moreover, as a special value of $h(A, B) = (6A+1)^{3}+14(6B-1)^{3}$,
we obtain another prime number $l = 419 = (-11)^{3}+14 \cdot 5^{3}$,
which is prime to $\{ P, b_{1}, c_{1} \}$.
Finally, in order to generate the coefficients $b_{0} = -1+6b$ and $c_{0}$,
we solve the equation
\[
	14^{2}(-1+6b) - 419c_{0} = \pm3^{k}.
\]
It has solutions $(b_{0}, c_{0}, \pm3^{k}) = (365, 13^2, 3^{6})$ with square $c_{0}$.
Therefore,
for every $m = 3, 5, 7$, the equation
\[
	(X^{3}+14^{2}Y^{3})(365X^{3}+419 \cdot 13^{2}Y^{3})(X^{2}+XY+Y^{2}) = 419^{m}Z^{8}
\]
defines a non-singular plane curve which violates the local-global principle.
However, its quotient by the automorphism $Z \to \zeta_{4}Z$
gives a hyperelliptic curve defined by
\[
	(X^{3}+14^{2}Y^{3})(365X^{3}+419 \cdot 13^{2}Y^{3})(X^{2}+XY+Y^{2}) = 419^{m}Z^{2}
\]
which has a $\Q$-rational point $[X : Y : Z] = [0 : 1 : 13 \cdot 14/419^{(m-1)/2}]$.

\subsection{Example for \cref{main_deg4} with $u = 79$}

Next, we construct an example for \cref{main_deg4} again for $n = 8$ and $u = 79$.
\footnote{
	There are many positive integers satisfying the whole of the conditions in \cref{main_deg4},
	say $u = 21, 35, 39, 79, 89, \dots$.
	Among them, $u = 79$ is the minimal one having no prime divisors $q = 3$ or $q \equiv 2 \bmod{3}$.
	Recall that, due to \cref{recipe_even},
	if $u$ has a prime divisor $q = 3$,
	then we need a congruence $Pb_{0}-lc_{0} = \pm 1$ stronger than $Pb_{0}-lc_{0} = \pm 3^{k}$.
	On the other hand,
	due to the proof of \cref{main_deg4} (cf.\ the proof of \cref{main_even} in \S3),
	if $u$ has a prime divisor $q \equiv 2 \bmod{3}$,
	then we need a congruence $b_{0} \equiv 1 \bmod{6q}$ stronger than $b_{0} \equiv 1 \bmod{6}$. 
	}
As a special value of $f(X, Y) = 158(3X-1)^{3}+(3Y+1)^{3}$,
we obtain a prime number $19751 = 158 \cdot 5^{3} + 1^{3}$ with $(b_{1}, c_{1}) = (5, 1)$.
Moreover, as a special value of $h(A, B) = (6A+1)^{3}+158(6B-1)^{3}$,
we obtain another prime number $l = 4919 = (-59)^{3}+158 \cdot 11^{3}$,
which is prime to $\{ P, b_{1}, c_{1} \}$.
Then, by solving the equation
\[
	158(1+6b) - 4919c_{0} = \pm3^{k}
\]	
with $b_{0} = 1+6b$,
we obtain a solution $(b_{0}, c_{0}, \pm3^{k}) = (271^{2}, 2359, -3^{5})$ with square $b_{0}$.
Therefore,
the equation
\[
	(X^{3}+158Y^{3})(271^{2}X^{3}+4919 \cdot 2359Y^{3})(5^{2}X^{2}+5XY+Y^{2}) = 4919^{2}Z^{8}
\]
defines a non-singular plane curve which violates the local-global principle.
However, its quotient by the automorphism $Z \to \zeta_{4}Z$
gives a hyperelliptic curve defined by
\[
	(X^{3}+158Y^{3})(271^{2}X^{3}+4919 \cdot 2359Y^{3})(5^{2}X^{2}+5XY+Y^{2}) = 4919^{2}Z^{2},
\]
which again has a $\Q$-rational point $[X : Y : Z] = [1 : 0 : 5 \cdot 271/4919]$.

\section{Appendix: Non-singularity of some Nguyen's sextic curves}

In this appendix,
we give the following proposition,
which is implicit in Nguyen's original article \cite{Nguyen_QJM}.

\begin{proposition} [{Non-singularity of some curves in \cref{Nguyen_QJM}}] \label{Nguyen_sextic}
Let $n = 2k$ be an even integer with $k \geq 1$.
Take $p, d, m, \text{and} \ \alpha$ as follows:
\begin{enumerate}
\item
$p$ is a prime number such that $p \equiv 1 \bmod{8}$.

\item
$d$ is an integer which is a quadratic non-residue in $\F_{p}^{\times}$ and prime to $n$.

\item
$m$ is an even integer such that $q := d^{2}+pm^{2}$ is a prime number.

\item
$\alpha$ is a rational number such that
$\alpha \in \Z_{l}$ for every prime divisor $l$ of $dp$
and $\alpha \neq 0, qp^{-k}, qd^{-k}, (m(d+p)-2q)((dp)^{k}-d^{k}-p^{k})^{-1}$.
\end{enumerate}
Set
$A = q-\alpha p^{k}$,
$B = q-\alpha d^{k}$,
and $C = m(d+p)-2q-\alpha((dp)^{k}-d^{k}-p^{k})$.
Let $\mathcal{C}^{(k)}_{p, d, m, \alpha}$ be a plane curve of degree $4k+2$
defined by the following homogeneous equation over $\Q$
\begin{align*}
	& pq^{2}X^{4k+2} + Y^{4k-2}(d(d+p)X^{2}-qY^{2})(pm^{2}(d+p)X^{2}-dqY^{2}) \\
	&\quad - Z^{2}(AX^{2k} + BY^{2k} + CX^{k}Y^{k} + \alpha Z^{2k})^{2} = 0.
\end{align*}
Suppose that there exists a prime divisor $l$ of $d$
such that $l \geq 5$ and $\alpha \equiv -m/3 \bmod{l}$.
Then, the curve $\mathcal{C}^{(1)}_{p, d, m, \alpha}$ is geometrically non-singular.
\end{proposition}

The above \cref{Nguyen_sextic} is a special case of the following \cref{general}.
Indeed, in the setting of \cref{Nguyen_sextic},
we have $\gcd(d, 2pqm(d+p)B\alpha) = 1$,
$2A+C \equiv -\alpha p^{1} + mp \equiv 4mp/3 \bmod{d}$,
and $2A-4B-C \equiv -3\alpha p^{1}-mp \equiv 0 \bmod{l}$ for some prime divisor $l$ of $d$.

\begin{proposition} \label{general}
Let $s, t,u, v, w, A, B, C, D \in \Z$.
Define a plane curve
$\mathcal{C}^{(k)} = \mathcal{C}^{(k)}_{s, t,u, v, w, A, B, C, D }$
of degree $4k+2$ by the following homogeneous equation
\[
	F(X, Y, Z)
	= sX^{4k+2} + Y^{4k-2}(tX^{2}+uY^{2})(vX^{2}+wY^{2})
	- Z^{2}(AX^{2k}+BY^{2k}+CX^{k}Y^{k}+DZ^{2k})^{2} = 0.
\]
Suppose that $w \neq 0$ and there exists a common prime divisor $l$ of $t$ and $w$
such that $l \nmid 6ksuv(2A+C)BD$ and $l \mid 2A-4B-C$.
Then, the curve $\mathcal{C}^{(1)}$ is geometrically non-singular.
\end{proposition}

Note that $t$ may be 0 but $s, u, v, w$ cannot be 0.

\begin{proof}
We prove it by contradiction.
Suppose that $[X : Y : Z] = [X_{0} : Y_{0} : Z_{0}]$ is a singular point on $\mathcal{C}$.
Then, we may assume that $X_{0}, Y_{0}, Z_{0} \in \overline{\Q}$.
Let $l$ be a common prime divisor of $t$ and $w$ in the assertion,
and $\lambda$ be a prime ideal of the number field $\Q(X_{0}, Y_{0}, Z_{0})$ lying above $l$.
Then, by dividing $X_{0}, Y_{0}, Z_{0}$
by an element $W \in \lambda \setminus \lambda^{2}$ if necessary,
we may assume that
$X_{0}, Y_{0}, Z_{0}$ are all $\lambda$-adic integers
and at least one of them is prime to $\lambda$.

Since
\begin{align*}
	\frac{\partial F}{\partial Z}
	&= -2Z(AX^{2k}+BY^{2k}+CX^{k}Y^{k}+DZ^{2k})^{2}
		- 2kDZ^{2k+1}(AX^{2k}+BY^{2k}+CX^{k}Y^{k}+DZ^{2k}) \\
	&= -2Z(AX^{2k}+BY^{2k}+CX^{k}Y^{k}+DZ^{2k})
		(AX^{2k}+BY^{2k}+CX^{k}Y^{k}+DZ^{2k} + 2kDZ^{2k}),
\end{align*}
the condition $(\partial F/\partial Z)(X_{0}, Y_{0}, Z_{0}) = 0$ implies that
\[
	Z_{0}(AX_{0}^{2k}+BY_{0}^{2k}+CX_{0}^{k}Y_{0}^{k}+DZ_{0}^{2k}) = 0
	\quad \text{or} \quad
	AX_{0}^{2k}+BY_{0}^{2k}+CX_{0}^{k}Y_{0}^{k}+DZ_{0}^{2k} = -2kDZ_{0}^{2k}.
\]

\begin{enumerate}
\item
First, suppose that $g := Z_{0}(AX_{0}^{2k}+BY_{0}^{2k}+CX_{0}^{k}Y_{0}^{k}+DZ_{0}^{2k}) = 0$.
Then, we have
\begin{align*}
	F(X_{0}, Y_{0}, Z_{0})
	&= sX_{0}^{4k+2} + Y_{0}^{4k-2}(tX_{0}^{2}+uY_{0}^{2})(vX_{0}^{2}+wY_{0}^{2}) \\
	&= sX_{0}^{4k+2} + Y_{0}^{4k-2}(tvX_{0}^{4}+(tw+uv)X_{0}^{2}Y_{0}^{2}+uwY_{0}^{4}) = 0.
\end{align*}
This shows that if $X_{0} = 0$, then since we assume that $u, w \neq 0$, we must have $Y_{0} = 0$.
However, the condition $g = 0$ with the assumption $l \nmid D$ (hence $D \neq 0$) implies that $Z_{0} = 0$, a contradiction.
Thus, $X_{0} \neq 0$.

Moreover, since $g = 0$,
we have
\[
	\frac{\partial F}{\partial X}(X_{0}, Y_{0}, Z_{0})
	= (4k+2)sX_{0}^{4k+1} + Y_{0}^{4k-2}(4tvX_{0}^{3} + (2tw+2uv)X_{0}Y_{0}^{2})
	= 0,
\]
hence
\begin{align*}
	&\quad (4k+2)F(X_{0}, Y_{0}, Z_{0}) - X_{0}\frac{\partial F}{\partial X}(X_{0}, Y_{0}, Z_{0}) \\
	&= Y_{0}^{4k-2}((4k-2)tvX_{0}^{4} + 4k(tw+uv)X_{0}^{2}Y_{0}^{2} + (4k+2)uwY_{0}^4)
		= 0.
\end{align*}

Now,
by taking into account of the conditions $X_{0} \neq 0$ and $t \equiv w \equiv 0 \bmod{l}$,
we have the following.
\[
\begin{cases}
	X_{0}^{-2}F(X_{0}, Y_{0}, Z_{0})
	\equiv sX_{0}^{4k} + uvY_{0}^{4k} \equiv 0 \bmod{\lambda}, \\
	(4k+2)F(X_{0}, Y_{0}, Z_{0}) - X_{0}\frac{\partial F}{\partial X}(X_{0}, Y_{0}, Z_{0})
	\equiv 4kuvX_{0}^{2}Y_{0}^{4k} \equiv 0 \bmod{\lambda}. \\
\end{cases}	
\]
Since we assume that $l \nmid 2kuv$,
the second congruence implies that
either $X_{0} \equiv 0 \bmod{\lambda}$ or $Y_{0} \equiv 0 \bmod{\lambda}$.
Moreover, since we assume that  $l \nmid suv$p
the first congruence implies that $X_{0} \equiv Y_{0} \equiv 0 \bmod{\lambda}$.
However, since we assume that $l \nmid D$,
the condition $g = 0$ implies that $Z_{0} \equiv 0 \bmod{\lambda}$, a contradiction.
This completes the proof under the condition $g = 0$ but for general $k$.
\footnote{
	In this step, we used the following assumptions:
	$w \neq 0$ (note that $l \mid w$) and 
	$l \nmid 2ksuvD$ for some common prime divisor $l$ of $t$ and $w$.
	}

\item
Next, suppose that $g \neq 0$,
and $h := AX_{0}^{2k}+BY_{0}^{2k}+CX_{0}^{k}Y_{0}^{k}+(2k+1)D = 0$.
Then, we have
\begin{align*}
	F(X_{0}, Y_{0}, Z_{0})
	&= sX_{0}^{4k+2} + Y_{0}^{4k-2}(tX_{0}^{2}+uY_{0}^{2})(vX_{0}^{2}+wY_{0}^{2})
		- Z_{0}^{2}(-2kDZ_{0}^{2k})^{2} \\
	&= sX_{0}^{4k+2} + Y_{0}^{4k-2}(tvX_{0}^{4}+(tw+uv)X_{0}^{2}Y_{0}^{2}+uwY_{0}^{4})
		- 4k^{2}D^{2}Z_{0}^{4k+2} = 0.
\end{align*}

Moreover,
since $h = 0$,
we have
\begin{align*}
	\frac{\partial F}{\partial X}(X_{0}, Y_{0}, Z_{0})
	&= (4k+2)sX_{0}^{4k+1} + Y_{0}^{4k-2}(4tvX_{0}^{3} + 2(tw+uv)X_{0}Y_{0}^{2}) \\
	&\quad + Z_{0}^{2} \cdot (-2kDZ_{0}^{2k}) \cdot (2kBY_{0}^{2k-1} + kCX_{0}^{k-1}Y_{0}^{k}) \\
	&= (4k+2)sX_{0}^{4k+1} + X_{0}Y_{0}^{4k-2}(4tvX_{0}^{2} + 2(tw+uv)Y_{0}^{2}) \\
	&\quad -2k^{2}DY_{0}^{k}Z_{0}^{2k+2} (2BY_{0}^{k-1} + CX_{0}^{k-1})
	= 0.
\end{align*}
Similarly,
we also have
\begin{align*}
	\frac{\partial F}{\partial Y}(X_{0}, Y_{0}, Z_{0})
	&= (4k-2)Y_{0}^{4k-3}(tvX_{0}^{4}+(tw+uv)X_{0}^{2}Y_{0}^{2}+uwY_{0}^{4}) \\
	&\quad +  Y_{0}^{4k-2}(2(tw+uv)X_{0}^{2}Y_{0} + 4uwY_{0}^{3}) \\
	&\quad + Z_{0}^{2} \cdot (-2kDZ_{0}^{2k}) \cdot (2kAX_{0}^{2k-1} + kCX_{0}^{k}Y_{0}^{k-1}) \\
	&= Y_{0}^{4k-3} ( (4k-2)tvX_{0}^{4} + 4k(tw+uv)X_{0}^{2}Y_{0}^{2} + (4k+2)uwY_{0}^{4} ) \\
	&\quad - 2k^{2}DX_{0}^{k}Z_{0}^{2k+2}(2AX_{0}^{k-1} + CY_{0}^{k-1})
		= 0.
\end{align*}

Now, by taking into account of the comdition $t \equiv w \equiv 0 \bmod{l}$,
we have
\[
\begin{cases}
	h = 0, \ \text{i.e.,} \ AX_{0}^{2k}+BY_{0}^{2k}+CX_{0}^{k}Y_{0}^{k}+(2k+1)DZ_{0}^{2k} = 0 \\
	F(X_{0}, Y_{0}, Z_{0})
	\equiv sX_{0}^{4k+2} + uvX_{0}^{2}Y_{0}^{4k} - 4k^{2}D^{2k}Z_{0}^{4k+2} \equiv 0 \bmod{\lambda} \\
	\frac{\partial F}{\partial X}(X_{0}, Y_{0}, Z_{0})
	\equiv (4k+2)sX_{0}^{4k+1} + 2uvX_{0}Y_{0}^{4k}
		- 2k^{2}DY_{0}^{k}Z_{0}^{2k+2}(2BY_{0}^{k-1} + CX_{0}^{k-1}) \equiv 0 \bmod{\lambda} \\
	\frac{\partial F}{\partial Y}(X_{0}, Y_{0}, Z_{0})
	\equiv 4kuvX_{0}^{2}Y_{0}^{4k-1} - 2k^{2}DX_{0}^{k}Z_{0}^{2k+2}(2AX_{0}^{k-1}+CY_{0}^{k-1}) \equiv 0 \bmod{\lambda}.
\end{cases}
\]
In particular,
from the assumption $l \nmid 2ksuvBD$,
we see that $X_{0}Z_{0} \not\equiv 0 \bmod{\lambda}$.
Moreover, by combining the congruences for $\partial F/\partial X$ and $\partial F/\partial Y$,
we can delete $Z_{0}$ as follows.
\begin{align*}
	&\quad (2AX_{0}^{k-1}+CY_{0}^{k-1})X_{0}^{k}\frac{\partial F}{\partial X}(X_{0}, Y_{0}, Z_{0})
		- (2BY_{0}^{k-1} + CX_{0}^{k-1})Y_{0}^{k}\frac{\partial F}{\partial Y}(X_{0}, Y_{0}, Z_{0}) \\
	&\equiv (4k+2)s(2AX_{0}^{k-1}+CY_{0}^{k-1})X_{0}^{5k+1}
		+ 2uv(2AX_{0}^{k-1}+CY_{0}^{k-1})X_{0}^{k+1}Y_{0}^{4k} \\
	&\quad - 4kuv(2BY_{0}^{k-1}+CX_{0}^{k-1})X_{0}^{2}Y_{0}^{5k-1} \\
	&\equiv (4k+2)s(2AX_{0}^{k-1}+CY_{0}^{k-1})X_{0}^{5k+1} \\
	&\quad + 2uvX_{0}^{2}Y_{0}^{4k}\left(
		2AX_{0}^{2k-2} - (2k-1)CX_{0}^{k-1}Y_{0}^{k-1} -4kBY_{0}^{2k-2}
			\right) \\
	&\equiv 0 \bmod{\lambda},
\end{align*}
hence
\begin{align*}
	&\quad (2k+1)s(2AX_{0}^{k-1}+CY_{0}^{k-1})X_{0}^{5k-1} \\
	&\quad	+ uvY_{0}^{4k}\left(
		2AX_{0}^{2k-2} - (2k-1)CX_{0}^{k-1}Y_{0}^{k-1} -4kBY_{0}^{2k-2}
			\right)
		\equiv 0 \bmod{\lambda}.
\end{align*}

Now, suppose that $k = 1$.
Then, the above congruence becomes
\[
	3s(2A+C)X_{0}^{4} + (2A-4B-C)uvY_{0}^{4} \equiv 0 \bmod{\lambda}.
\]
Thus, if $l \nmid 3(2A+C)$ and $l \mid 2A-4B-C$
as we assumed,
then we must have $X_{0} \equiv 0 \bmod{\lambda}$, a contradiction.
\footnote{
	In this step, we used the following assumptions:
	$k = 1$,
	$l \nmid 6ksuv(2A+C)BD$,
	and $l \mid 2A-4B-C$ for some common prime divisor $l$ of $t$ and $w$.
	}
\end{enumerate}
This completes the proof (at least for $k = 1$).
\end{proof}

\section*{Acknowledgements}

The author would like to thank Yoshinori Kanamura and Yosuke Shimizu
for fruitful discussions on the local-global principle.
In fact,
a question from Kanamura is the motivation of this research
and the previous joint work with Shimizu forms the foundation of this research.
This article cannot realize without valuable discussions with them.
The author would like to thank also
Ken-ichi Bannai, Hiroyasu Izeki, Yoshinori Kanamura, Masato Kurihara, Takaaki Tanaka,
and Yusuke Tanuma for careful reading of the draft of this article and many valuable comments.
Finally, the author would express his sincere gratitude to anonymous referee
for her/his careful reading of our manuscript and many valuable suggestions.

\end{document}